\documentclass[a4paper, 11pt]{article}
\usepackage{mathrsfs}
\usepackage{mathrsfs}
\usepackage{amsmath,amssymb}
\usepackage{amsfonts}
\usepackage[T1]{fontenc}
\usepackage[latin9]{inputenc}
\usepackage{amsthm}
\usepackage{graphicx}
\usepackage{amsmath}

\usepackage{amsthm}
\usepackage{array}
\usepackage{cases}
\usepackage{enumerate,enumitem,todonotes}
\usepackage{thmtools}
\usepackage{thm-restate}
\usepackage{setspace}  
\usepackage{changes}

\makeatletter
\def\th@plain{%
  \itshape 
}
\makeatother

\makeatletter
\renewenvironment{proof}[1][\proofname]{\par
  \pushQED{\qed}%
  \normalfont \topsep6\p@\@plus6\p@\relax
  \trivlist
  \item[\hskip\labelsep
        \bfseries
    #1\@addpunct{.}]\ignorespaces
}{%
  \popQED\endtrivlist\@endpefalse
}
\makeatother

\numberwithin{equation}{section}

\newtheorem{thm}{Theorem}[section]
\newtheorem{cor}[thm]{Corollary}
\newtheorem{claim}[thm]{Claim}
\newtheorem{lemma}[thm]{Lemma}
\newtheorem{prop}{Proposition}

\newtheorem{conj}[thm]{Conjecture}

\numberwithin{equation}{section}

\usepackage[varg]{txfonts}
\usepackage{graphicx, epsfig, subfigure}
\usepackage{enumerate}
\usepackage[square, numbers, sort&compress]{natbib}
\ifx\pdfoutput\undefined
 \usepackage[dvipdfm,%
  pdfstartview=FitH,
  bookmarks=true,%
  bookmarksnumbered=true,
  bookmarksopen=true,
  plainpages=false,%
  pdfpagelabels,%
  colorlinks=true,
  linkcolor=blue,
  citecolor=blue,%
  urlcolor=black,
  pdfborder=001]{hyperref}
  \AtBeginDvi{}  
\else
 \usepackage[pdftex,%
  pdfstartview=FitH,
  bookmarks=true,%
  bookmarksnumbered=true,
  bookmarksopen=true,
  plainpages=false,%
  pdfpagelabels,%
  colorlinks=true,
  linkcolor=blue,
  citecolor=blue,%
  urlcolor=black,
  pdfborder=001]{hyperref}%
\fi

\usepackage[left]{lineno}
\RequirePackage[normalem]{ulem} 
\RequirePackage{color}\definecolor{RED}{rgb}{1,0,0}\definecolor{BLUE}{rgb}{0,0,1} 
\providecommand{\DIFaddbegin}{} 
\providecommand{\DIFaddend}{} 
\providecommand{\DIFdelbegin}{} 
\providecommand{\DIFdelend}{} 
\providecommand{\DIFaddbeginFL}{} 
\providecommand{\DIFaddendFL}{} 
\providecommand{\DIFdelbeginFL}{} 
\providecommand{\DIFdelendFL}{} 
\newcommand{\DIFscaledelfig}{0.5}
\RequirePackage{settobox} 
\RequirePackage{letltxmacro} 
\newsavebox{\DIFdelgraphicsbox} 
\newlength{\DIFdelgraphicswidth} 
\newlength{\DIFdelgraphicsheight} 
\LetLtxMacro{\DIFOincludegraphics}{\includegraphics} 
\newcommand{\DIFaddincludegraphics}[2][]{{\color{blue}\fbox{\DIFOincludegraphics[#1]{#2}}}} 
\newcommand{\DIFdelincludegraphics}[2][]{
\sbox{\DIFdelgraphicsbox}{\DIFOincludegraphics[#1]{#2}}
\settoboxwidth{\DIFdelgraphicswidth}{\DIFdelgraphicsbox} 
\settoboxtotalheight{\DIFdelgraphicsheight}{\DIFdelgraphicsbox} 
\scalebox{\DIFscaledelfig}{
\parbox[b]{\DIFdelgraphicswidth}{\usebox{\DIFdelgraphicsbox}\\[-\baselineskip] \rule{\DIFdelgraphicswidth}{0em}}\llap{\resizebox{\DIFdelgraphicswidth}{\DIFdelgraphicsheight}{
\setlength{\unitlength}{\DIFdelgraphicswidth}
\begin{picture}(1,1)
\thicklines\linethickness{2pt} 
{\color[rgb]{1,0,0}\put(0,0){\framebox(1,1){}}}
{\color[rgb]{1,0,0}\put(0,0){\line( 1,1){1}}}
{\color[rgb]{1,0,0}\put(0,1){\line(1,-1){1}}}
\end{picture}
}\hspace*{3pt}}} 
} 
\LetLtxMacro{\DIFOaddbegin}{\DIFaddbegin} 
\LetLtxMacro{\DIFOaddend}{\DIFaddend} 
\LetLtxMacro{\DIFOdelbegin}{\DIFdelbegin} 
\LetLtxMacro{\DIFOdelend}{\DIFdelend} 
\DeclareRobustCommand{\DIFaddbegin}{\DIFOaddbegin \let\includegraphics\DIFaddincludegraphics} 
\DeclareRobustCommand{\DIFaddend}{\DIFOaddend \let\includegraphics\DIFOincludegraphics} 
\DeclareRobustCommand{\DIFdelbegin}{\DIFOdelbegin \let\includegraphics\DIFdelincludegraphics} 
\DeclareRobustCommand{\DIFdelend}{\DIFOaddend \let\includegraphics\DIFOincludegraphics} 
\LetLtxMacro{\DIFOaddbeginFL}{\DIFaddbeginFL} 
\LetLtxMacro{\DIFOaddendFL}{\DIFaddendFL} 
\LetLtxMacro{\DIFOdelbeginFL}{\DIFdelbeginFL} 
\LetLtxMacro{\DIFOdelendFL}{\DIFdelendFL} 
\DeclareRobustCommand{\DIFaddbeginFL}{\DIFOaddbeginFL \let\includegraphics\DIFaddincludegraphics} 
\DeclareRobustCommand{\DIFaddendFL}{\DIFOaddendFL \let\includegraphics\DIFOincludegraphics} 
\DeclareRobustCommand{\DIFdelbeginFL}{\DIFOdelbeginFL \let\includegraphics\DIFdelincludegraphics} 
\DeclareRobustCommand{\DIFdelendFL}{\DIFOaddendFL \let\includegraphics\DIFOincludegraphics} 
\begin{document}
\title{\LARGE  
Every subcubic graph is packing $(1,1,2,2,3)$-colorable
}

\author{
Xujun Liu\thanks{Department of Foundational Mathematics, Xi'an Jiaotong-Liverpool University, Suzhou, Jiangsu Province, 215123, China, xujun.liu@xjtlu.edu.cn; the research of X. Liu was supported by the Natural Science Foundation of the Jiangsu Higher Education Institutions of China (Grant No. 22KJB110025) and the Research Development Fund RDF-21-02-066 of Xi'an Jiaotong-Liverpool University.} \and
Xin Zhang\thanks{School of Mathematics and Statistics, Xidian University, Xi'an, Shaanxi Province, 710126, China, xzhang@xidian.edu.cn;
the research of X. Zhang was
supported by the Natural Science Basic Research Program of Shaanxi Province (Nos.\,2023-JC-YB-001, 2023-JC-YB-054).} \and 
Yanting Zhang\thanks{Mathematical Institute, University of Oxford, Oxford, OX2 6GG, United Kingdom, yanting.zhang@maths.ox.ac.uk; the work was partially done when Y. Zhang was an undergraduate student at Xi'an Jiaotong-Liverpool University, Suzhou, China.}
}


\maketitle

\begin{abstract}
\baselineskip 0.60cm

For a sequence $S=(s_1, \ldots, s_k)$ of non-decreasing integers, a packing $S$-coloring of a graph $G$ is a partition of its vertex set $V(G)$ into $V_1, \ldots, V_k$ such that for every pair of distinct vertices $u,v \in V_i$, where $1 \le i \le k$, the distance between $u$ and $v$ is at least $s_i+1$. The packing chromatic number, $\chi_p(G)$, of a graph $G$ is the smallest integer $k$ such that $G$ has a packing~$(1,2, \ldots, k)$-coloring. Gastineau and Togni asked an open question ``Is it true that the $1$-subdivision ($D(G)$) of any subcubic graph $G$ has packing chromatic number at most $5$?'' and later Bre\v sar, Klav\v zar, Rall, and Wash conjectured that it is true. 

In this paper, we prove that every subcubic graph has a packing $(1,1,2,2,3)$-coloring and it is sharp due to the existence of subcubic graphs that are not packing $(1,1,2,2)$-colorable. As a corollary of our result, $\chi_p(D(G)) \le 6$ for every subcubic graph $G$, improving a previous bound ($8$) due to Balogh, Kostochka, and Liu in 2019, and we are now just one step away from fully solving the conjecture. 

\vspace{3mm}\noindent \emph{Keywords}: packing colorings; packing chromatic number; subcubic graphs; $1$-subdivision 
\end{abstract}

\baselineskip 0.60cm

\section{Introduction}
For a sequence $S=(s_1, \ldots, s_k)$ of non-decreasing integers, a packing $S$-coloring of a graph $G$ is a partition of its vertex set $V(G)$ into $V_1, \ldots, V_k$ such that for every pair of distinct vertices $u,v \in V_i$, where $1 \le i \le k$, the distance between $u$ and $v$ is at least $s_i+1$. The packing chromatic number, $\chi_p(G)$, of a graph $G$ is defined to be the smallest integer $k$ such that $G$ has a packing~$(1,2, \ldots, k)$-coloring. The concept of packing $S$-coloring was first introduced by Goddard and Xu~\cite{GX1} and is now a very popular topic in graph coloring. Moreover, its edge counterpart was recently studied by Gastineau and Togni~\cite{GT1}, Hocquard, Lajou, and Lu\v zar~\cite{HLL1}, Liu, Santana, and Short~\cite{LSS1}, as well as Liu and Yu~\cite{LY1}.

The notion of packing chromatic number was introduced by Goddard, Hedetniemi, Hedetniemi, Harris, and Rall~\cite{GHHHR1} in 2008 under the name broadcast chromatic number, and it was motivated by a frequency assignment problem in broadcast networks. The concept has drawn the attention of many researchers recently (e.g., see~\cite{BKL1,BKL2,BF1,BKRW1,BKRW2,FKL1,GT2,GHHHR1, KL1, LLRY1,MT1}). The question whether every cubic graph has a bounded packing chromatic number was first asked by Goddard et al~\cite{GHHHR1} and discussed in many papers (e.g., see~\cite{BKL1,BKRW1,BKRW2,GT2}). The $1$-subdivision of a graph $G$, denoted by $D(G)$, is obtained from $G$ by replacing every edge with a path of two edges. Gastineau and Togni~\cite{GT2} asked the open question whether it is true that the subdivision of any subcubic graph is packing $(1,2,3,4,5)$-colorable
and later Bre\v sar, Klav\v zar, Rall, and Wash~\cite{BKRW2} conjectured it is true.

\begin{conj}[Bre\v sar, Klav\v zar, Rall, and Wash~\cite{BKRW2}]\label{conj1}
The $1$-subdivision of every subcubic graph is packing $(1,2,3,4,5)$-colorable.    
\end{conj}

Balogh, Kostochka, and Liu~\cite{BKL1} answered the question of Goddard et al~\cite{GHHHR1} in the negative using the probabilistic method and later Bre\v sar and Ferme~\cite{BF1} provided an explicit construction. 
In contrast, Balogh, Kostochka, and Liu~\cite{BKL2} proved that the packing chromatic number of the $1$-subdivision of subcubic graphs is bounded by $8$. Furthermore, Conjecture~\ref{conj1} has been confirmed for many subclasses of subcubic graphs. In particular, Bre\v sar, Klav\v zar, Rall, and Wash~\cite{BKRW2} proved it for generalized prism of a cycle, Liu, Liu, Rolek, and Yu~\cite{LLRY1} showed it for subcubic planar graphs with girth at least $8$, Kostochka and Liu~\cite{KL1} confirmed it for subcubic outerplanar graphs, and Mortada and Togni~\cite{MT1} recently extended this class by including each subcubic $3$-saturated graph that has no adjacent heavy vertices.


Gastinue and Togni~\cite{GT2} proved the following statement, which is invaluable for proving Conjecture~\ref{conj1}.

\begin{prop}[Gastineau and Togni~\cite{GT2}]\label{tool}
Let $G$ be a graph and $(s_1, \ldots, s_k)$ be a sequence of non-decreasing positive integers. If $G$ is packing $(s_1, \ldots, s_k)$-colorable, then $D(G)$ is packing $(1, 2s_1+1, \ldots, 2s_k+1)$-colorable.   
\end{prop}

Gastineau and Togni~\cite{GT2} showed that the Petersen graph has no packing $(1,1,k,k')$-colorings when $k,k' \ge 2$. Indeed, the maximum size of the union of two independent sets in the Petersen graph is $7$ and the diameter of the Petersen graph is $2$. Bre\v sar, Klav\v zar, Rall, and Wash~\cite{BKRW2} proved that the $1$-subdivision of the Petersen graph is packing $(1,2,3,4,5)$-colorable. By Proposition~\ref{tool}, if one can show every subcubic graph except the Petersen graph has a packing $(1,1,2,2)$-coloring, then Conjecture~\ref{conj1} is confirmed.

Many other packing $S$-colorings have also been studied. In particular, Gastineau and Togni~\cite{GT2} proved that every subcubic graph is packing $(1,1,2,2,2)$-colorable and packing $(1,2,2,2,2,2,2)$-colorable. Balogh et al~\cite{BKL2} showed that every subcubic graph has a packing $(1,1,2,2,3,3,k)$ with color $k \ge 4$ used at most once and every $2$-degenerate subcubic graph has a packing $(1,1,2,2,3,3)$-coloring. Cranston and Kim~\cite{CK1} showed that every cubic graph except the Petersen graph has a packing $(2,2,2,2,2,2,2,2)$-coloring. Thomassen~\cite{T1} and independently Hartke, Jahanbekam, and Thomas~\cite{HJT1} proved that every cubic planar graph is packing $(2,2,2,2,2,2,2)$-colorable.

In this paper, we prove that every subcubic graph has a packing $(1,1,2,2,3)$-coloring.

\begin{thm}\label{mainresult}
Every subcubic graph $G$ has a packing $(1,1,2,2,3)$-coloring.
\end{thm}

Our result is also sharp due to the fact that the Petersen graph is not packing $(1,1,2,2)$-colorable. By Theorem~\ref{mainresult} and Proposition~\ref{tool}, a packing $(1,1,2,2,3)$-coloring of $G$ implies a packing $(1,3,3,5,5,7)$-coloring of $D(G)$. Therefore, $\chi_p(D(G)) \le 6$ for every subcubic graph $G$, improving the previous bound ($8$) of Balogh et al~\cite{BKL2}, and we are now just one step away from fully solving Conjecture~\ref{conj1}. 

\begin{cor}
Let $G$ be a subcubic graph. Then $\chi_p(D(G)) \le 6$.
\end{cor}

\vspace{-3mm}
\section{Proof of Theorem \ref{mainresult}}

We may assume that the graph $G$ is cubic and connected since every connected subcubic graph is a proper subgraph of some larger connected cubic graph. Take two disjoint independent sets $I_1$ and $I_2$ such that
\begin{equation}\label{condition-1}
   |I_1| + |I_2| \text{ is maximum among all choices of } I_1, I_2.
\end{equation}
Subject to Condition~\ref{condition-1}, we further take $I_1, I_2$ such that
\begin{equation}\label{condition-2}
  \text{the number of connected components in } G-I_1-I_2 \text{ is minimum}.
\end{equation}

Let $G' = G[V(G) - I_1 - I_2]$ and define the graph $H_{I_1, I_2}$ to be the graph $H$ with $V(H) = V(G) - I_1 - I_2$ and $E(H) = \{v_1v_2 \text{ }|\text{ } d_G(v_1,v_2) \le 2, v_1, v_2 \in V(H)\}$. We use the abbreviation $H$ to denote $H_{I_1, I_2}$ if the sets $I_1$ and $I_2$ are clear from the context.
The corresponding graph $G(H)$ of $H$ in $G$ is the graph with 
$V(G(H))=V(H)\cup \{u \text{ }|\text{ }u\in V(G) \setminus V(H),~v_1,v_2\in V(H),\text{ and }uv_1,uv_2\in E(G)\}$
and
$E(G(H))=\{v_1v_2 \text{ }|\text{ } v_1v_2\in E(G), v_1, v_2 \in V(H)\}\cup \{uv_1,uv_2 \text{ }|\text{ } v_1v_2\not\in E(G), uv_1,uv_2\in E(G), u\in V(G) \setminus V(H), v_1, v_2 \in V(H)\}$.
Essentially, $G(H)$ shows how each component of $H$ is connected in $G$. Denote $G'$ the {\em red graph} and each vertex in $I_1 \cup I_2$ {\em black}. For the addition and subtraction of subscripts within the set $\{1,2, \ldots, k\}$, we perform these operations modulo $k$. 

\begin{lemma}\label{maxdegree}
  $\Delta(G') \le 1.$
\end{lemma}

\begin{proof}
Suppose not, i.e., there is a vertex $u$ of degree at least two in $G'$. 
Let $\{u_1,u_2\}\subseteq N_{G'}(u)$.
If $N_{G'}(u) \setminus \{u_1,u_2\}=\emptyset$, then we may assume 
$N_{G}(u) \setminus \{u_1,u_2\}=\{u_3\}$ and $u_3\in I_1$.
Since $\{u_1,u_2\}\cap (I_1\cup I_2)=\emptyset$, no matter whether $N_{G'}(u) \setminus \{u_1,u_2\}$ is empty or not, we can add $u$ to $I_2$ to obtain a contradiction with Condition~\ref{condition-1}.
\end{proof}






By Lemma~\ref{maxdegree}, we have the following corollary.

\begin{cor}\label{cor-p1-p2}
A connected component in $G'$ is either a $P_1$ or a $P_2$.   
\end{cor}

We now pay attention to the structures between two $P_2$s (of $G'$) in $G$. By Lemma~\ref{maxdegree}, two red $P_2$s cannot be adjacent in $G$. Furthermore, we show at most one $P_2$ can be included in a connected component of $H$.

\begin{lemma}\label{onep2}
At most one red $P_2$ can be included in a connected component of $H$.
\end{lemma}

\begin{proof}
We first show two red $P_2$s cannot be connected by a vertex in $I_1 \cup I_2$ (i.e., at distance one in $H$).

\begin{claim}\label{basecase}
Two red $P_2$s cannot be connected by a vertex in $I_1 \cup I_2$ (i.e., at distance one in $H$).
\end{claim}

\begin{proof}
Suppose not, i.e., two red $P_2$s, $u_1u_2$ and $v_1v_2$, are connected by a vertex $w_1 \in I_1$ with $u_2w_1,v_1w_1 \in E(G)$. Let $N(u_2) = \{u_1, w_1, u_2'\}$ and $N(v_1) = \{w_1, v_2, v_1'\}$. By Lemma~\ref{maxdegree}, $u_2',v_1' \in I_1 \cup I_2.$ We may assume $u_2',v_1' \in I_2$ since otherwise we can add $u_2$ or $v_1$ to $I_2$, which contradicts Condition~\ref{condition-1}. We remove $w_1$ from $I_1$ and add $u_2,v_1$ to $I_1$ to increase the size of $I_1 \cup I_2$, which is again a contradiction with Condition~\ref{condition-1}.
\end{proof}

We are ready to prove the lemma by induction. The base case is already shown in Claim~\ref{basecase}. Assume two red $P_2$s cannot be at distance at most $k-1$ in $H$, where $k \ge 2$. We now show two red $P_2$s cannot be at distance $k$ in $H$. Let $u_1u_2, v_1v_2$ be two $P_2$s in $H$ that are at distance $k$. Let $x_1, \ldots, x_{k-1} \in V(G')$ and $w_1,\ldots, w_k \in I_1 \cup I_2$ such that $u_1u_2w_1x_1w_2\ldots x_{k-1}w_kv_1v_2$ is a path in $G$. We may assume $w_1 \in I_1$. Let $N(u_2) = \{u_1, w_1, u_2'\}$. We know $u_2' \in I_2$ since otherwise we can add $u_2$ to $I_2$ to increase the size of $I_1 \cup I_2$, which is a contradiction with Condition~\ref{condition-1}. We remove $w_1$ from $I_1$ and add $u_2$ to $I_1$, creating a new red $P_2$ $w_1x_1$ which is at distance $k-1$ from $v_1v_2$ in $H$. This is a contradiction with the inductive hypothesis.   
\end{proof}

We turn our attention to the structures between red $P_1$s (of $G'$) in $G$.

\begin{lemma}\label{noC_1}
If three red vertices are joined to the same vertex in $I_1 \cup I_2$ (we call such a configuration $C_1$), then they form a triangle component by themselves in $H$ (see Figure~\ref{C_1C_2} left picture).
\end{lemma}

\begin{proof}
Let $u_1,v_1,w_1$ be three red vertices in $V(G')$ with their common neighbour $u \in I_1$. We first note that $u_1v_1 \notin E(G)$. Otherwise, say $N(u_1) = \{u,v_1,u_2\}$ and $N(v_1) = \{u,u_1,v_2\}$. If $u_2$ or $v_2$ belongs to $I_1$, then we add $u_1$ or $v_1$ to $I_2$ respectively, which is a contradiction with Condition~\ref{condition-1}. Thus, we assume $u_2,v_2 \in I_2$. Now we move $u$ from $I_1$ to $I_2$, and add $v_1$ to $I_1$, which again contradicts Condition~\ref{condition-1}. Similarly, $u_1w_1, v_1w_1 \notin E(G)$.
Let $N(u_1) = \{u, u_2, u_3\}$, $N(v_1) = \{u, v_2, v_3\}$, and $N(w_1) = \{u, w_2, w_3\}$. 
Suppose first that $u_2$ is red. If $u_3\in I_1$, then we add $u_1$ to $I_2$.
If $u_3\in I_2$, then we move $u$ from $I_1$ to $I_2$, and add $u_1$ to $I_1$.
In each case we obtain contradiction with Condition~\ref{condition-1}. Hence $u_2\in I_1\cup I_2$, and by symmetry $u_3,v_2,v_3,w_2,w_3\in I_1\cup I_2$.


We may assume $u_2 \in I_1$ and $u_3 \in I_2$ since otherwise, say both $u_2,u_3 \in I_2$, we move $u$ from $I_1$ to $I_2$, and add $u_1$ to $I_1$, which is a contradiction with Condition~\ref{condition-1}. Similarly, we assume $v_2,w_2 \in I_1$ and $v_3, w_3 \in I_2$. Let $N(u_2) = \{u_1, u_4,u_5\}$. By symmetry, we only need to show $u_4,u_5 \in I_1 \cup I_2$. We first know $u_4,u_5$ cannot be both in $V(G')$ since otherwise we move $u,u_2$ from $I_1$ to $I_2$, and add $u_1$ to $I_1$. This contradicts Condition~\ref{condition-1}. Therefore, we may assume $u_4 \in V(G')$ and $u_5 \in I_2$. However, we move $u_2$ to $V(G')$, move $u$ to $I_2$, and add $u_1$ to $I_1$. The number of components in $G'$ is decreased by one, which contradicts Condition~\ref{condition-2}.
\end{proof}

\begin{figure}[ht]
\vspace{-8mm}
\begin{center}
  \includegraphics[scale=0.66]{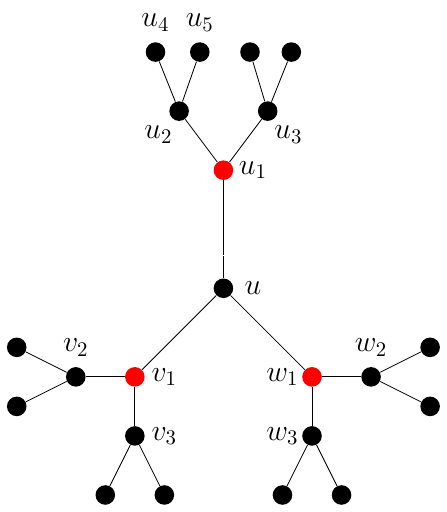} \hspace{35mm}
  \includegraphics[scale=0.58]{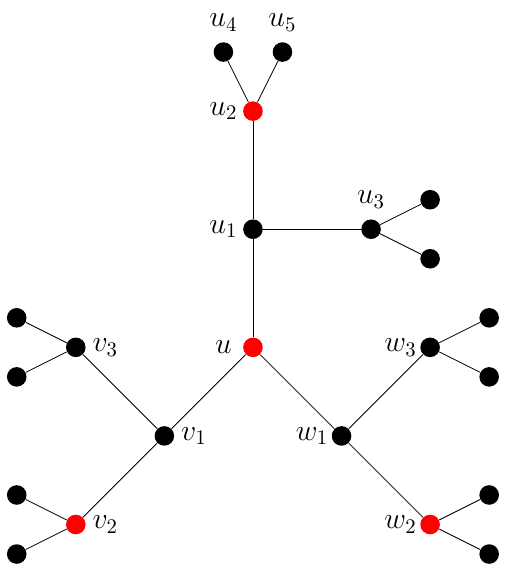} 
\caption{Configurations $C_1$ and $C_2$.}\label{C_1C_2}
\end{center}
\vspace{-8mm}
\end{figure}

Let $C_2$ be the configuration that three red $P_1$s are joining to the same red $P_1$ via a vertex in $I_1 \cup I_2$ (see Figure~\ref{C_1C_2} right picture). We show that configuration $C_2$ does not exist in $G$.

\begin{lemma}\label{noC_2}
Configuration $C_2$ does not exist in $G$.
\end{lemma}

\begin{proof}
Suppose to the contrary that there is a configuration $C_2$ in $G$ (see Figure~\ref{C_1C_2} right picture). We have three red $P_1$s, $u_2,v_2,w_2$, joining to the same red $P_1$, $u$, via vertices $u_1,v_1,w_1 \in I_1 \cup I_2$ respectively. By Lemma~\ref{noC_1}, $u_3,v_3,w_3 \in I_1 \cup I_2$. If all of $u_1,v_1,w_1 \in I_1$ (or $I_2$), then we can add $u$ to $I_2$ and the size of $I_1 \cup I_2$ is increased, which is a contradiction with Condition~\ref{condition-1}. Therefore, we may assume $u_1 \in I_1$ and $v_1,w_1 \in I_2$. However, we move $u_1$ to $V(G')$ and add $u$ to $I_1$, which is a contradiction with Condition~\ref{condition-2} as the number of components in $G'$ is decreased by one.
\end{proof}

We show the maximum degree of $H$ is bounded and prove some additional properties of $H$.

\begin{lemma}\label{propertyH}
$\Delta(H) \le 3$. Furthermore, if $u_1$ is a $3$-vertex in $H$, then it must belong to a red $P_2$ (cannot be a red $P_1$), say $u_1u_2$, in $G$ and $u_2$ must be a $1$-vertex in $H$. 
\end{lemma}

\begin{proof}
Let $u_1\in V(G')$.
By Corollary \ref{cor-p1-p2}, $u_1$ is either an endpoint of a red $P_2$ or a $P_1$ by itself. 

In the former case, say this red $P_2$ is $u_1u_2$. Let $N(u_1) = \{u_2, u_3, u_4\}$ and $N(u_2) = \{u_1, u_5, u_6\}$. By Lemma~\ref{maxdegree}, $u_3,u_4,u_5,u_6 \in I_1 \cup I_2$. Let $N(u_3) = \{u_1, u_7, u_8\}$, $N(u_4) = \{u_1, u_9, u_{10}\}$, $N(u_5) = \{u_2, u_{11}, u_{12}\}$, $N(u_6) = \{u_2, u_{13}, u_{14}\}$. By Lemma~\ref{noC_1}, $u_7$ and $u_8$ cannot be both in $G'$, and $u_9$ and $u_{10}$ cannot be both in $G'$. Thus, $u_1$ has degree at most $3$ in $H$. We may assume that $u_3 \in I_1$ and $u_4 \in I_2$ since otherwise, say $u_3,u_4 \in I_1$, we can add $u_1$ to $I_2$ and it contradicts Condition~\ref{condition-1}. Similarly, we may assume $u_5 \in I_1$ and $u_6 \in I_2$.

\begin{claim}\label{p2structure}
If $u_7 \in V(G')$, then both $u_{13}$ and $u_{14}$ must be in $I_1 \cup I_2$. By symmetry, if $u_9 \in V(G')$, then both $u_{11}$ and $u_{12}$ must be in $I_1 \cup I_2$.    
\end{claim}

\begin{proof}
Suppose to the contrary that $u_{13} \in V(G')$. By Lemma~\ref{onep2}, each of $u_7$ and $u_{13}$ is a red $P_1$. By Lemma~\ref{noC_1}, $u_{14} \in I_1$. We move $u_3$ and $u_6$ to $V(G')$, and add $u_1$ to $I_1$ and $u_2$ to $I_2$. This is a contradiction with Condition~\ref{condition-2} since the number of components in $G'$ is dropped by one. 
\end{proof}

\noindent Suppose now that $u_1$ is a $3$-vertex in $H$. It follows that $|\{u_7,u_8\}\cap V(G')|=|\{u_9,u_{10}\}\cap V(G')|=1$.
Assume, without loss of generality, that $u_7,u_9\in V(G')$. Claim \ref{p2structure} implies that $u_2$ has degree 1 in $H$.

In the latter case, say $N(u_1) = \{u_2, u_3, u_4\}$ with $u_2,u_3,u_4 \in I_1 \cup I_2$. Let $N(u_2) = \{u_1, u_5, u_6\}$, $N(u_3) = \{u_1, u_7, u_8\}$, and $N(u_4) = \{u_1, u_9, u_{10}\}$. By Lemma~\ref{noC_1}, $|\{u_5,u_6,u_7,u_8,u_9,u_{10}\}\cap V(G')|\leq 3$ and thus $u_1$ has degree at most three in $H$. If $u_1$ has degree three, then Lemma~\ref{noC_1} implies $|\{u_5,u_6\}\cap V(G')|=|\{u_7,u_8\}\cap V(G')|=|\{u_9,u_{10}\}\cap V(G')|=1$. 
Assume, without loss of generality, that $u_5,u_7,u_9\in V(G')$. We then have a configuration $C_2$ and it contradicts Lemma~\ref{noC_2}.
\end{proof}

The proof of Claim~\ref{p2structure} actually implies the following lemma.

\begin{lemma}\label{diagonal}
   Let $N(u_1) = \{u_2,u_3,u_4\}$ and $N(u_2) = \{u_1, u_5,u_6\}$. If $u_1u_2$ is a red $P_2$ with $u_3,u_5 \in I_1$ and $u_4, u_6 \in I_2$, then it is impossible to have two distinct red $P_1$s so that one is in $N(u_3)\setminus \{u_1\}$ and the other is in $N(u_6)\setminus \{u_2\}$.\hfill $\square$
   
\end{lemma}

\begin{lemma} \label{no.cycle+bar}
 Each component of $H$ is not isomorphic to a cycle with one vertex adjacent to a leaf.
\end{lemma}

\begin{proof}
Suppose not, i.e., $H$ has a cycle with one vertex adjacent to a leaf.
Let the cycle be $u_1u_2 \ldots u_ku_1$ and let $x_1$ be the leaf adjacent to $u_1$.
By Lemma~\ref{propertyH}, $x_1u_1$ is a red $P_2$ in $G$. 
Since there are no other red $P_2$s in the cycle by Lemma~\ref{onep2}, 
each $u_i$ with $2\leq i\leq k$ is a red $P_1$ in $G$. 
Now the definition of $H$ implies that there is a cycle, say $u_1w_1u_2w_2 \cdots u_kw_ku_1$, in $G$.
Let $N(u_i) = \{w_{i-1},w_i,x_i\}$ and $N(w_i) = \{u_i,u_{i+1},y_i\}$, where $1 \le i \le k$.
By Lemma~\ref{onep2} and~\ref{noC_1}, each $x_i$ with $i\neq 1$ and each $y_i$ are in $I_1\cup I_2$.
We may assume $w_1 \in I_1$ and $w_k \in I_2$ since otherwise we can add $u_1$ to $I_1$ or $I_2$, which contradicts Condition~\ref{condition-1}.
It follows that there is an $i$ with $2 \le i \le k$ such that $w_{i-1} \in I_1$ and $w_i \in I_2$.
If $x_i\in I_1$, then let $I_2:=I_2\cup\{u_i\}\setminus \{w_i\}$.
If $x_i\in I_2$, then let $I_1:=I_1\cup\{u_i\}\setminus \{w_{i-1}\}$.
In each case we obtain a contradiction with Condition \ref{condition-2}.
\end{proof}

By Lemmas~\ref{propertyH} and \ref{no.cycle+bar}, we conclude that 
each component of $H$ is a tree or an even cycle or an odd cycle.
Clearly, $H$ is 3-colorable.
Let $h$ be a proper $3$-coloring of $H$ using colors $A,B,C$ such that 

\textbf{(i)} the color
$C$ is used exactly once on each odd cycle component of $H$, and more precisely,

\textbf{(ii)} if
$u_1u_2$ is an edge of an odd cycle component such that $u_1,u_2\in V(G')$, 
then we arbitrarily choose a vertex from $u_1$ and $u_2$, and color its other neighbor on the cycle with $C$. 

We now complete the proof of Theorem~\ref{mainresult}. 
Since each vertex of $G$ is either in $I_1\cup I_2$ or colored with $A,B$, or $C$,
we construct a coloring $f$ of $G$ by assigning
vertices in $I_1$ with color $1_a$,  
vertices in $I_2$ with color $1_b$,
vertices colored by $A$ with color $2_a$,
vertices colored by $B$ with color $2_b$,
and vertices colored by $C$ with color $3$.
Since $I_1$ and $I_2$ are independent sets in $G$, vertices with color $1_a$ or $1_b$ forms a $1$-independent set respectively.
By the definitions of $H$ and the colorings $h$ and $f$, vertices with color $2_a$ or $2_b$ forms a $2$-independent set respectively.
At last, it is sufficient to show that vertices with color $3$ forms a $3$-independent set, and therefore, $f$ is a packing $(1,1,2,2,3)$-coloring of $G$.


Suppose not, i.e., there are two vertices $u,v$ with $f(u)=f(v)=3$ and $d_G(u,v) \le 3$ (denoted by a {\em $3$-$3$ conflict}). By the coloring assignment rules of $h$ and $f$, $u$ and $v$ are in different components of $H$. Let $S_1,S_2$ be two components of $H$ such that $u \in S_1$ and $v\in S_2$. 
Since the color $3$ ($C$) is only used on the odd cycle components of $H$,
we may assume that $S_1$ and $S_2$ are odd cycles. Let the cycle $S_1$ be $u_1u_2 \ldots u_ku_1$, $k \ge 3$.
We discuss different cases where a $3$-$3$ conflict can occur.

\textbf{Case 1:} 
The corresponding cycle of $S_1$ in $G$ has a red $P_2$, say, $u_1u_2$.

The cycle $S_1$ of $H$ corresponds to a cycle of $G$, say, 
$u_1u_2w_2\cdots u_kw_ku_1$.
Let $N(u_1) = \{u_2,w_k,x_1\}$, $N(u_2) = \{u_1,w_2,x_2\}$, $N(u_i) = \{w_{i-1},w_i,x_i\}$ for $3 \le i \le k$, and $N(w_i) = \{u_i, u_{i+1}, y_i\}$ for $2 \le i \le k$. By Corollary~\ref{cor-p1-p2}, Lemma~\ref{onep2} and~\ref{noC_1}, all $x_i$s, $y_i$s, and $w_i$s are in $I_1\cup I_2$.
By the rule \textbf{(ii)} of $h$ and by the definition of $f$, we may assume $f(u_3) = 3$, i.e., $u:=u_3$.

By Condition \ref{condition-1}, we assume $w_k \in I_1$ and $x_1 \in I_2$.
If $k=3$, then we claim $w_2\in I_1$, and consequently (by Condition~\ref{condition-1}) $x_2\in I_2$ and $x_3\in I_2$.
Suppose not, i.e., $w_2\in I_2$. Assume, without loss of generality, that $x_3\in I_2$.  
Now we can reassign $I_1:=I_1\cup \{u_1, u_3\}\setminus \{w_3\}$ and $I_2:=I_2$. This is a contradiction with Condition~\ref{condition-1}. If $k\geq 5$, then by Lemma \ref{diagonal} we have $w_2\in I_1$ and $x_2\in I_2$. 

We claim for every $2 \le i \le k$, $w_i \in I_1, x_i,y_i \in I_2$, and the cycle $u_1u_2w_2\cdots u_kw_ku_1$ has no chord. This is true for $k=3$ and thus we assume $k \ge 5$. 
Suppose that there is an $i$ with $w_{i-1} \in I_1$ and $w_i \in I_2$, where $3 \le i \le k-2$. If $x_i \in I_1$, 
then we reassign $I_1:=I_1$ and $I_2:=I_2\cup \{u_i\}\setminus \{w_i\}$.
If $x_i\in I_2$,
then we reassign $I_1:=I_1\cup \{u_i\}\setminus \{w_{i-1}\}$ and $I_2:=I_2$.
In either case we obtain a contradiction with Condition \ref{condition-2}, since the number of components in $H$ is dropped by one. Hence, $w_i \in I_1$ for every $2 \le i \le k$. It follows that
$x_i,y_i \in I_2$ for every $2 \le i \le k$, and there is no chord in the cycle.


\begin{claim}\label{2ndneighbour}
$N(x_i)\setminus \{u_i\} \subseteq I_1$ and $N(y_i) \setminus \{ w_i\} \subseteq I_1$, where $1 \le i \le k$. 
\end{claim}

\begin{proof}
Since $S_1$ is a component of $H$, for each 
$1 \le i \le k$, vertices in $N(x_i)\setminus \{u_i\}$ cannot be red. Thus, $N(x_i)\setminus \{u_i\} \subseteq I_1$ since $x_i$s are all in $I_2$.
For each $w_iu_{i+1}$ with $2 \le i \le k$, it can become a red $P_2$ via reassigning 
$I_1:=I_1\cup \{u_2,\ldots, u_i\}\setminus \{w_2,\ldots, w_i\}$ and
$I_2:=I_2$. 
Note that this switch operation does not violate Conditions \ref{condition-1} and \ref{condition-2}, but after this operation we can apply Lemma \ref{diagonal} to show that $N(y_i) \setminus \{ w_i\} \subseteq I_1$. 
\end{proof}

Now we lock the position of the vertex $v$. Recall that $v$ and $u$ (actually $u_3$) form a 3-3 conflict.
Let $N(x_3) = \{u_3, z_1, z_2\}$, $N(z_1) = \{x_3, z_3, z_4\}$, and $N(z_2) = \{x_3, z_5, z_6\}$. Since $d_G(u_3,v) \le 3$, 
$v \in \{z_3, z_4, z_5, z_6\}$ by Claim~\ref{2ndneighbour}.
Without loss of generality, assume $v=z_3$.
Next we show at most two of $z_3,z_4,z_5,z_6$ are in $G'$.

\begin{claim}\label{u3structure}
At most two of $z_3,z_4,z_5,z_6$ are in $G'$.
\end{claim}

\begin{proof}
If $\{z_3,z_4,z_5,z_6\} \subseteq V(G')$, then
we reassign $I_1:=I_1\cup \{x_3\}\setminus \{z_1,z_2\}$ and 
$I_2:=I_2\cup \{z_1,z_2,u_3\}\setminus \{x_3\}$, which contradicts Condition \ref{condition-1}.
If three of $z_3,z_4,z_5,z_6$ are in $G'$, say, $z_3,z_4,z_5 \in G'$ and $z_6 \in I_2$, then we reassign
$I_1:=I_1\cup \{x_3\}\setminus \{z_1,z_2\}$ and 
$I_2:=I_2\cup \{z_1,u_3\}\setminus \{x_3\}$.
This drops the number of compoments in $H$ by one and thus violates Condition \ref{condition-2}.
Hence, $|\{z_3,z_4,z_5,z_6\} \cap V(G')|\leq 2$.
\end{proof}



According to Claim ~\ref{u3structure}, we only need to consider two cases up to symmetry, i.e., $z_3,z_4 \in G'$ and $z_5,z_6 \in I_2$, or $z_3,z_5 \in G'$ and $z_4,z_6 \in I_2$.

If $z_3,z_4 \in G'$ and $z_5,z_6 \in I_2$, then we claim that $z_3z_4$ itself is a component of $H$, which implies $v\neq z_3$, a contradiction. 
We reassign $I_1:=I_1$ and $I_2:=I_2\cup \{u_3\}\setminus \{x_3\}$. This operation adds $x_3$ to $G'$,
and does not violate Conditions~\ref{condition-1} and~\ref{condition-2}. Now applying Lemma~\ref{noC_1} we come to the required conclusion.

Suppose now that $z_3,z_5 \in G'$ and $z_4,z_6 \in I_2$.
Since $z_3$ ($v$) is a vertex of $S_2$ and it is not part of a $P_2$ in $G'$ according to rule \textbf{(ii)} of the coloring $h$, $N(z_3)\setminus \{z_1\}\subseteq I_1\cup I_2$, and furthermore, each vertex of $N(z_3)\setminus \{z_1\}$ has a neighbor, besides $z_3$, in $G'$ (see Figure~\ref{finalcases} left picture).
In this case, we reassign $I_1:=I_1$ and $I_2:=I_2\cup \{u_3\}\setminus \{x_3\}$. This operation adds $x_3$ to $G'$,
and does not violate Conditions~\ref{condition-1} and~\ref{condition-2}, but forms a configuration $C_2$, contradicting Lemma~\ref{noC_2}.



\begin{figure}[ht]
\vspace{-8mm}
\begin{center}
  \includegraphics[scale=0.52]{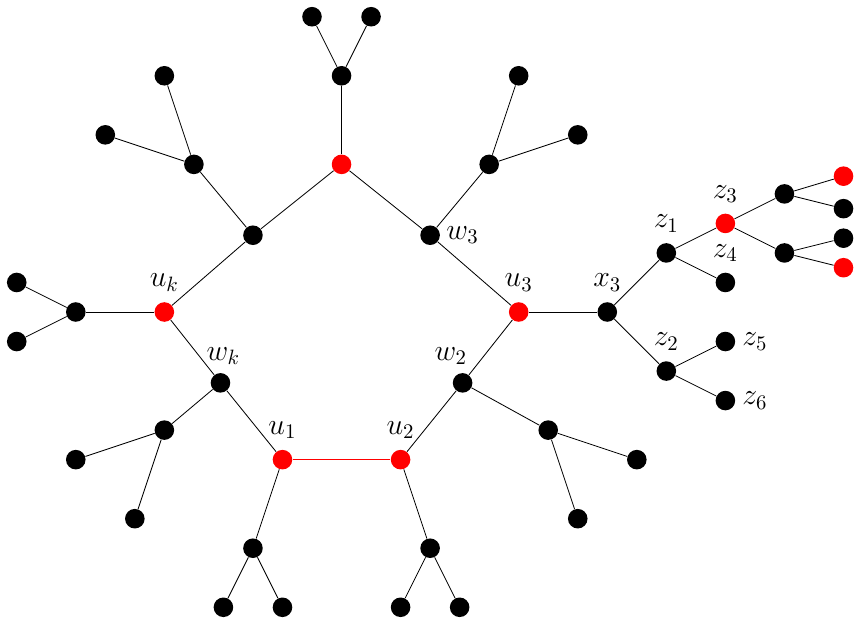} \hspace{3mm}
  \includegraphics[scale=0.52]{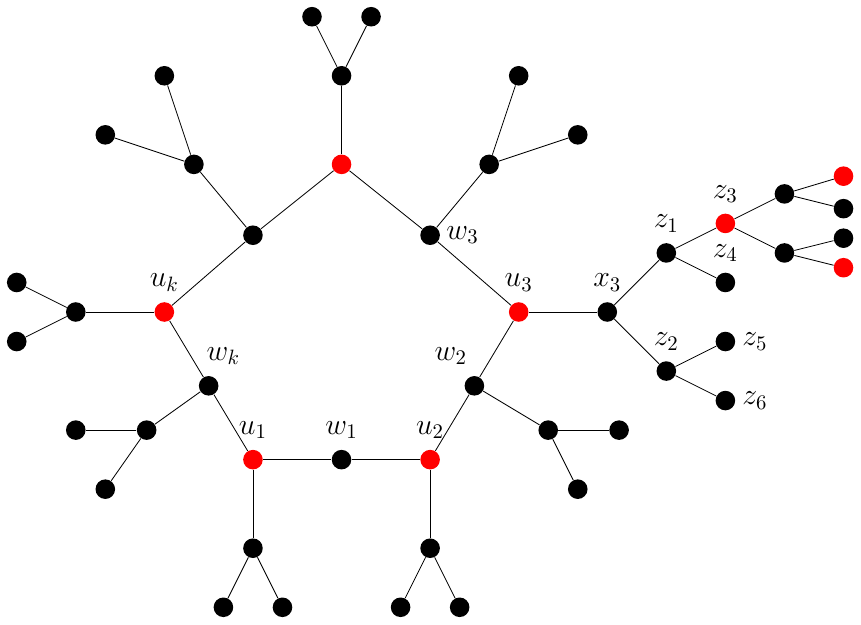} 
\caption{Case 1 and Case 2.}\label{finalcases}
\end{center}
\vspace{-8mm}
\end{figure}

\textbf{Case 2:} The corresponding cycles of both $S_1$ and $S_2$ in $G$ have no red $P_2$. 

The cycle $S_1$ of $H$ corresponds to a cycle $u_1w_1u_2w_2 \ldots u_kw_ku_1$ in $G$, where $u_i \in V(G')$ for $1 \le i \le k$. For $1 \le i \le k$, let $N(u_i) = \{w_i, w_{i-1},x_i\}$ and $N(w_i) = \{u_i, u_{i+1}, y_i\}$. We assume $f(u_3) = 3$, i.e., $u=u_3$, and  $w_1 \in I_1$.


We claim for every $1 \le i \le k$, $w_i \in I_1$ and $x_i,y_i \in I_2$, and there is no chord in the cycle $u_1w_1u_2w_2 \ldots u_kw_ku_1$. 
Suppose not, i.e., there is an $i$ with $w_{i-1} \in I_1$ and $w_i \in I_2$, where $2 \le i \le k$. If $x_i \in I_1$, then we switch $u_i,w_i$ between in $I_2$ and in $G'$. This is a contradiction with Condition~\ref{condition-2}. 
If $x_i \in I_2$, then we switch $u_i, w_{i-1}$ between in $I_1$ and in $G'$. This is again a contradiction with Condition~\ref{condition-2}. 
Therefore, $w_i \in I_1$ and $y_i \in I_2$ for each $1 \le i \le k$. Furthermore, $x_i\in I_2$, where $1 \le i \le k$, since otherwise we add $u_i$ to $I_2$, which violates Condition~\ref{condition-1}. By Lemma~\ref{noC_1}, the cycle $u_1w_1u_2w_2 \ldots u_kw_ku_1$ is chordless.

\begin{claim}\label{2ndneighbour-case2}
$N(x_i) \setminus \{u_i\} \subseteq I_1$ and $N(y_i) \setminus \{w_i\} \subseteq I_1$, where $1 \le i \le k$. 
\end{claim}

\begin{proof}
By Lemma~\ref{noC_2}, for each $1 \le i \le k$ we have $N(x_i) \setminus \{u_i\} \subseteq I_1$. Moreover, we reassign
$I_1:=I_1\cup \{u_1,\ldots, u_k\}\setminus \{w_1,\ldots, w_k\}$ and
$I_2:=I_2$. This operation does not violate Conditions~\ref{condition-1} and~\ref{condition-2}, but now we can apply  Lemma~\ref{noC_2} to conclude that
$N(y_i) \setminus \{w_i\} \subseteq I_1$ for each $1 \le i \le k$. 
\end{proof}


Recall that $v$ and $u$ (actually $u_3$) form a 3-3 conflict. Let $N(x_3) = \{u_3, z_1, z_2\}$, $N(z_1) = \{x_3, z_3, z_4\}$, and $N(z_2) = \{x_3, z_5, z_6\}$  (see Figure~\ref{finalcases} right picture).
Since $d_G(u_3,v) \le 3$, 
$v \in \{z_3, z_4, z_5, z_6\}$ by Claim \ref{2ndneighbour-case2}.
Applying almost the same proof with Case 1 starting from Claim~\ref{u3structure}, we complete the proof.





\section{Concluding Remarks}

The study of packing $(1,1,2,2)$-coloring of subcubic graphs is a crucial approach to prove Conjecture~\ref{conj1}. We observe that a packing $(1,\ldots,1,2,\ldots, 2)$-coloring can be viewed as an intermediate coloring between a proper coloring (packing $(1, \ldots, 1)$-coloring) and a square coloring (packing $(2, \ldots, 2)$-coloring). Methods used in proving results of proper coloring and square coloring (e.g., see~\cite{FKL1, HM1, T1}) can be useful.

We feel one might approach the problem "every subcubic graph except the Petersen graph has a packing $(1,1,2,2)$-coloring" by providing a more detailed analysis on the odd cycle components.
Furthermore, adding a condition regarding the odd cycles in $H$, such as "the number of odd cycles in $H$ is minimized", maybe helpful. 
We conclude this paper by posting the following conjecture.

\begin{conj}
Every subcubic graph except the Petersen graph has a packing $(1,1,2,2)$-coloring.
\end{conj}


\end{document}